\documentclass[preprint,12pt]{elsarticle}

\usepackage{amsmath,latexsym,amscd,amsbsy,amssymb,amsfonts,amsthm,fleqn,leqno,
euscript, graphicx, texdraw, pb-diagram}
\usepackage{mathrsfs}
\usepackage{color}
\journal{Topology and its Applications}

\newtheorem{thm}{Theorem}[section]
\newtheorem{pro}[thm]{Proposition}
\newtheorem{lem}[thm]{Lemma}

\newtheorem{cor}[thm]{Corollary}

\def\leukfrac#1/#2{\leavevmode
               \kern.1em
                \raise.9ex\hbox{\the\scriptfont0 ${}_#1$}
                \hskip -1pt\kern-.1em
                /\kern-.15em\lower.10ex\hbox{\the\scriptfont0 ${}_#2$}}

\newtheorem{dfn}{Definition}[section]

\theoremstyle{definition}

\newtheorem{re}[thm]{Remark}

\newtheorem{question}[thm]{Question}

\theoremstyle{remark}
\newtheorem{claim}{Claim}

\DeclareMathOperator{\inte}{int}

\makeatletter
\def\Int{\mathop{\operator@font Int}\nolimits}
\makeatother

\begin{document}
\begin{frontmatter}
\title
{Strongly locally homogeneous generalized continua of finite cohomological dimension}

\author[a]{A. Karassev\fnref{1}}
\ead{alexandk@nipissingu.ca}
\affiliation[a]{{Department of Computer Science and Mathematics, Nipissing University},
            {100 College Drive},
            {North Bay},
            {P.O. Box 5002,P1B 8L7},
           {ON},
            {Canada}}
            \fntext[l]{Partially supported by NSERC Grant 257231-15}

\author[b]{P. Krupski\corref{cor1}}
\ead{pawel.krupski@pwr.edu.pl}
\affiliation[b]{organization={Wrocław University of Science and Technology, Faculty of Pure and Applied Mathematics},
            addressline={Wybrzeże Wyspiańskiego 27, 50-370 Wrocław},
                       country={Poland}}
            \cortext[cor1]{Corresponding author}

\author[c]{V. Todorov}
\ead{tt.vladimir@gmail.com}
\affiliation[c]{organization={Department of Mathematics, UACG},
            addressline={1 H. Smirnenski blvd., 1046 Sofia},
                       country={Bulgaria}}

\author[d]{V.  Valov\fnref{2}}
\ead{veskov@nipissingu.ca}
\affiliation[d]{organization={Department of Computer Science and Mathematics, Nipissing University},
            addressline={100 College Drive},
            city={North Bay},
            postcode={P.O. Box 5002,P1B 8L7},
            state={ON},
            country={Canada}}
             \fntext[2]{Partially supported by NSERC Grant 261914-19}

\begin{abstract}
The notion of a $V^n$-continuum was introduced by Alexandroff \cite{ps} as a generalization of the concept of $n$-manifold. In this note we consider the cohomological analogue of $V^n$-continuum and prove that any strongly locally homogeneous generalized continuum $X$ with  cohomological dimension $\dim_G X=n$ is  a generalized  $V^n$-space with respect to the cohomological dimension $\dim_G$. In particular, every strongly locally homogeneous continuum of covering dimension $n$ is a $V^n$-continuum in the sense of Alexandroff. This provides a partial answer to a question raised in \cite{tv}.

 An analog of the Mazurkiewicz theorem that no subset of covering dimension $\le n-2$ cuts any region of the Euclidean $n$-space is also obtained for      strongly locally homogeneous generalized continua $X$ of cohomological dimension $\dim_G X=n$.
\end{abstract}


\end{frontmatter}

\markboth{}{$V^n$-continua}



\section{Introduction}
By a \emph{space} we mean a locally compact separable metric space, unless  other\-wise stated, and  each of its connected open subsets is called a {\em region} or a  {\em generalized continuum}.     {\em Maps} are continuous mappings. The covering dimension of a space X is denoted by $\dim X$.  \v{C}ech  cohomology groups $H^n(X;G)$ with coefficients in an
Abelian group $G$ are considered everywhere below.

 The \emph{cohomological dimension} $\dim_G X$ of a space $X$ is the largest number $n$ such that there exists a closed subset $A\subset X$ with $$H^n(X,A;G)\neq 0.$$ Equivalently,  $\dim_G X\leq n$ if and only if for any closed pair $A\subset B$ in $X$ the homomorphism
$j_{X,A}^n:H^{n}(B;G)\to H^{n}(A;G)$, generated by the inclusion $A\hookrightarrow B$, is surjective, see \cite{dy}.

 If $X$ is a finite-dimensional space, then for every $G$ we have $\dim_GX\leq\dim_{\mathbb Z}X=\dim X$~\cite{ku}.

 The notion of a $V^n$-continuum was introduced by Alexandroff \cite{ps} as a generalization of the concept of $n$-manifold. A continuum $X$ is a {\em $V^n$-continuum} if for every two closed disjoint subsets $P$ and $Q$ of $X$, both having non-empty interiors $\inte P$, $\inte Q$, there exists an open cover $\omega$ of $X$ such that
 no partition $C\subset X$ between $P$ and $Q$ admits an $\omega$-map into a space $Y$ with $\dim Y\leq n-2$ (by an $\omega$-map on $C$ we mean an $\omega|C$-map, where $\omega|C=\{U\cap C: U\in\omega\}$).      We extend the notion of $V^n$-continua as follows.
\begin{dfn}
  A space $X$ is a {\em  $V^n(G)$-space} if for any two closed disjoint sets $P,Q$ in $X$ with  $\inte P\neq\varnothing\neq\inte Q$ there exists an open cover $\omega$ of $X$ such that no partition in $X$ between $P$ and $Q$ admits an $\omega$-map into a space $Y$ of cohomological dimension $\dim_G Y\le
 n-2$.
\end{dfn}
Compact $V^n(G)$-spaces were defined in \cite{kktv} under the name of  $V^n$-continua with respect to the class of all spaces with
$\dim_G\leq n-2 $  and studied in \cite{tv} as Alexandroff manifolds with respect to the class.
 One can easily observe that each  $V^n(G)$-space is connected, so it is a generalized continuum.

 Recall that a space $X$ is \emph{homogeneous} if for every $x,y\in X$ there is a homeomorphism $h:X\to X$ such that $h(x)=y$.
A space $X$ is {\em strongly locally homogeneous} if every point $x\in X$ has a basis of open neighborhoods $U$ such that for every $y,z\in U$ there is a homeomorphism $h$ on $X$ such that $h(y)=z$ and $h$ is the identity on $X\setminus U$. 

Basic examples of strongly locally homogeneous spaces are the Euclidean space $\mathbb R^n$ and the Menger universal $n$-dimensional continuum $\mathcal M_n$ \cite[Theorem 3.2.2]{be}.

 Every region of a strongly locally homogeneous space is homogeneous,  strongly locally homogeneous, and  it is also known to be locally connected~\cite{ba}.
 
 The question whether every homogeneous $ANR$-continuum $X$ of dimension $\dim X=n$ is a $V^n$-continuum was raised in \cite[Question 1.2]{tv}. In this note we prove the following theorem which provides a partial answer to that question.
\begin{thm}\label{t1}
Every strongly locally homogeneous generalized continuum $X$ with $\dim_GX=n$ is a  $V^n(G)$-space.
\end{thm}
Since $\dim_G\mathbb R^n= \dim_G\mathcal M_n = n$  for any Abelian group $G$, and both Euclidean $n$-manifolds and Menger $\mathcal M_n$-manifolds  are strongly locally homogeneous, we get the following corollary.
\begin{cor}
Connected Euclidean $n$-manifolds and Menger $\mathcal M_n$-manifolds are $V^n(G)$-spaces for any Abelian group $G$.
\end{cor}

It was established in \cite[Theorem 4]{kv} that every region $\Lambda$ of a homogeneous and locally connected space $X$ with $\dim_G\Lambda=n$ cannot be cut by an $F_\sigma$-set $M$ of dimension $\dim_G M\leq n-2$. For strongly locally homogeneous spaces we get the following theorem.

\begin{thm}\label{c1}
Let $X$ be a strongly locally homogeneous generalized continuum  with $\dim_GX=n$. Then for every pair $U,V\subset X$ of disjoint, nonempty, open sets,
 and every set $M\subset X$ with $\dim_GM\leq n-2$, there is a continuum $K\subset X\backslash M$ joining $U$ and $V$.
\end{thm}
It is well known that finite-dimensional homogeneous $ANR$-continua share many properties with Euclidean manifolds, see \cite{vv}. This fact and
Theorem~\ref{t1} suggest the following question, the positive solution of which would answer \cite[Question 1.2]{tv}.

\begin{question}
Let $X$ be a homogeneous $ANR$-continuum with $\dim X=n$. Is it true that $X$ is strongly locally homogeneous? 
\end{question}

\section{Proofs of Theorems~\ref{t1}  and~\ref{c1}}
If $\omega$ is an open cover of $X$ and $F\subset X$ is closed, for every closed $C\subset X$ we consider a canonical map $\kappa_{\omega_{C}}:(C,C\cap F)\to (N_{\omega|C},N_{\omega|(C\cap F)})$, where $N_{\omega|C}$ and
$N_{\omega|(C\cap F)}$ are the nerves of the covers $\omega|C$ and $\omega|(C\cap F)$, respectively (see, e.g.,~\cite[\S 11]{es}). The following notion was introduced in \cite{tv}.

Let $P$ and $Q$ be disjoint, nonempty, open subsets of  a continuum $X$, and $F\subset X$ be closed. We say that the pair $(X,F)$ is {\em $K_G^n$-connected between $P$ and $Q$} if there exists an open cover $\omega$ of $X$ such that the following condition holds for every partition $C$ in $X$ between $P$ and $Q$:

 the map $\kappa_{\omega_{C}}\colon (C,C\cap F)\to (N_{\omega|C},N_{\omega|(C\cap F)})$ generates a non-trivial homomorphism $$\kappa_{\omega_{C}}^*:H^{n-1}(N_{\omega|C},N_{\omega|(C\cap F)};G)\to H^{n-1}(C,C\cap F;G).$$ If, in addition, there is  $\mathrm{e}\in H^{n-1}(N_\omega,N_{\omega|(Y\cap F)})$ such that $$\kappa_{\omega_{C}}^*(i^{\ast}_{\omega_{C}}(\mathrm{e}))\neq 0$$  for every partition $C$ in $X$ between $P$ and $Q$, the pair $(X,F)$ is called {\em strongly $K^n_G$-connected between $P$ and $Q$}.
  Here, $Y=X\backslash(P\cup Q)$ and $$i_{\omega_{C}}:(N_{\omega|C},N_{\omega|(C\cap F)})\to (N_\omega,N_{\omega|(Y\cap F)})$$ is the inclusion.

Further, $(X,F)$ is said to be a {\em $K^n_G$-manifold} ({\em strong $K^n_G$-manifold}) if it is $K_G^n$-connected (resp., strongly $K^n_G$-connected) between any two nonempty open disjoint sets $P,Q\subset X$.

The notion of a (strong) $K^n_G$-manifold is justified by the following result which was actually established in \cite[Theorem 1]{ku1}.
\begin{pro}
Every compactum $X$ (not necessarily metrizable) with $\dim_GX=n$ contains a pair $(Y,F)$ of closed sets with $\dim_GY=n$ such that
$F$ is nowhere dense in $Y$ and $(Y,F)$ is a strong $K^n_G$-manifold.
\end{pro}
  We need  an analogous statement for locally compact spaces.

\begin{lem}\label{le1}
Every space $X$ with $\dim_G X=n$ contains  a strong $K^n_G$-manifold $(\Phi,F)$, where $\Phi$ is a compactum with $\dim_G\Phi=n$ and $F$ is a closed nowhere dense subset of  $\Phi$. In particular,   $\Phi$ is
a $V^n(G)$-space.
\end{lem}
\begin{proof}
Since $X$ is a countable union of compact sets, there exists a compactum $Y\subset X$ with $\dim_GY=n$ (otherwise, by the countable sum theorem for $\dim_G$, $\dim_G X\leq n-1$). Then, by Proposition 2.1, there is a pair $F\subset\Phi$ of compact sets in $Y$ with $\dim_G\Phi=n$ such that $F$ is nowhere dense in $\Phi$ and $(\Phi,F)$ is a strong $K^n_G$-manifold. Finally, according to~\cite[Proposition 2.2]{tv}, $\Phi$ is
a $V^n(G)$-space.
\end{proof}

\begin{re} According to~\cite[Theorem 2.6]{kktv}, every compactum $Y$ with $\dim_GY=n$ contains a compact $V^n(G)$-space $\Phi$ such that $\dim_G\Phi=n$.  Lemma~\ref{le1} is a stronger statement which will allow us to prove Theorem~\ref{c1}.
\end{re}

\begin{lem}\label{le2}
Let $X$ be a strongly locally homogeneous generalized continuum and $\Phi\subset X$ be a compactum containing at least two points. Suppose that $U$ and $V$ are nonempty open subsets of $X$ with $U\cap V=\varnothing$. Then there exists a homeomorphism $h:X\to X$ such that $h(\Phi)\cap U\neq\varnothing\neq h(\Phi)\cap V $.
\end{lem}
\begin{proof}
For every $x\in X$, let $\mathcal B_x$ be a basis at $x$ consisting of open sets $W\subset X$ such that for all $y,z\in W$ there is a homeomorphism $h$ on $X$ with $h(y)=z$ and such that $h$ is the identity outside $W$.

Since $X$ is homogeneous, we can assume that $\Phi\cap U\neq\varnothing$ and choose $x_U\in \Phi\cap U$. Let $\Gamma$ be the set of all $x\in X$ such that $x\neq x_U$ and there is a homeomorphism $h_x:X\to X$ with $x\in h_x(\Phi)$ and $h_x(x_U)=x_U$. Take $x\in\Phi\backslash\{x_U\}$. 
Obviously, the identity homeomorphism on $X$ witnesses that $x\in\Gamma$, so $\Gamma\neq\varnothing$.
\begin{claim}\label{cl1}
 $\Gamma$ is open in $X\backslash\{x_U\}$.
\end{claim}
Indeed, let $x\in\Gamma$.  So, there exists a homeomorphism $h_x$ on $X$ such that
$x\in h_x(\Phi)$ and $h_x(x_U)=x_U$. Choose  $W_x\in\mathcal B_x$  with $x_U\not\in W_x$.
Then for every $y\in W_x$ there is a homeomorphism $h_{xy}$ on $X$ with $h_{xy}(x)=y$ and $h_{xy}(z)=z$ for all $z\not\in W_x$. The composition $h_y=h_{xy}\circ h_x$ is a homeomorphism  such that $y\in h_y(\Phi)$ and $h_y(x_U)=x_U$. This shows that $W_x\subset\Gamma$, hence $\Gamma$ is open in $X\backslash\{x_U\}$.

\begin{claim}\label{cl2}
 $\Gamma$ is closed in $X\backslash\{x_U\}$.
\end{claim}
To show Claim~\ref{cl2}, let $\{x_k\}_{k=1}^\infty$ be a sequence in $\Gamma$ converging to a point $x_0\in X\backslash\{x_U\}$.
Since $x_0\neq x_U$, we
take  $W_{x_0}\in\mathcal B_{x_0}$ with $x_U\not\in W_{x_0}$ and let $x_k\in W_{x_0}$. There exist homeomorphisms $h_{x_k}$ and $h_{x_kx_0}$ on $X$ such that
$$x_k\in h_{x_k}(\Phi),\quad h_{x_k}(x_U)=x_U,\quad
h_{x_kx_0}(x_k)=x_0$$ and $h_{x_kx_0}$ is the identity outside $W_{x_0}$. Evidently, the homeomorphism $h_{x_0}=h_{x_kx_0}\circ h_{x_k}$ satisfies the conditions $x_0\in h_{x_0}(\Phi)$ and $h_{x_0}(x_U)=x_U$. Hence, $x_0\in\Gamma$, which completes the proof of Claim~\ref{cl2}.

Since $X\backslash\{x_U\}$ is a connected set \cite{kr},  $\Gamma=X\backslash\{x_U\}$ by Claims~\ref{cl1} and~\ref{cl2}. So, every $x\in V$ belongs to $\Gamma$. This means that there is a homeomorphism $h$ on $X$ such that $h(\Phi)$ meets both sets $U$ and $V$.
\end{proof}
\begin{re}
Another proof of Lemma \ref{le2} follows from the fact that every strongly locally homogeneous space is 2-homogeneous \cite{ba}. This means that if $A,B$ are 2-points sets in $X$, then there is a homeomorphism $h$ on $X$ with $h(A)=B$. To apply this fact, choose $A,B$ any 2-points sets such that $A\subset\Phi$ and $B$ meets both $U$ and $V$.   
\end{re}
\textit{Proof of Theorem~\ref{t1}}

Suppose $X$ is not a  $V^n(G)$-space. Hence, there are closed disjoint sets $P, Q\subset X$, both having nonempty interiors, such that for every open cover $\omega$ of $X$ there exists a partition $C_\omega$ in $X$ between $P$ and $Q$ admitting an $\omega$-map into a space $Y_\omega$ with $\dim_GY_\omega\leq n-2$.

According to Lemma~\ref{le1}, $X$ contains a compact $V^n(G)$-space $\Phi$.
By Lemma~\ref{le2}, there is a homeomorphism $h:X\to X$ such that $h(\Phi)$ meets both $\inte P$ and $\inte Q$.

Since $h(\Phi)$ is a $V^n(G)$-space, there is an open cover $\omega_0$ of $h(\Phi)$ such that no partition in $h(\Phi)$ between the sets $h(\Phi)\cap P$ and $h(\Phi)\cap Q$ admits an $\omega_0$-map into a space of dimension $\dim_G\leq n-2$. Extend $\omega_0$ to an open cover $\omega_0^*$ of $X$ and observe that, for the partition $C_{\omega_0^*}$,  the set $C_{\omega_0^*}\cap h(\Phi)$ is a partition in $h(\Phi)$ between
$h(\Phi)\cap P$ and $h(\Phi)\cap Q$ which admits an $\omega_0$-map into the space  $Y_{\omega_0^*}$ with $\dim_GY_{\omega_0^*}\leq n-2$, a contradiction. $\Box$

\

\textit{Proof of Theorem~\ref{c1}}

Suppose $U$ and $V$ are nonempty, disjoint, open subsets of $X$, and let $M\subset X$ be a set of dimension $\dim_GM\leq n-2$. Considering smaller subsets if necessary, we can assume that the closures $\overline U$ and $\overline V$ are also disjoint. By  Lemma~\ref{le1}, there is a compact pair $(\Phi,F)$ in $X$ with $\dim_G\Phi=n$ such that $F$ is nowhere dense in $\Phi$ and $(\Phi,F)$ is a strong $K^n_G$-manifold.
By Lemma~\ref{le2}, there is a homeomorphism $h$ on $X$ such that $h(\Phi)\cap U\neq\varnothing\neq h(\Phi)\cap V$. Then
$h(\Phi)$ is a continuum joining $U$ and $V$. So,  if $M\cap h(\Phi)=\varnothing$, we are done. If $M\cap h(\Phi)\neq\varnothing$, we can apply
\cite[Theorem 3.1]{tv} (see Theorem~\ref{t3} in the Appendix)  to find a continuum $K\subset h(\Phi)\backslash M$ joining $U\cap h(\Phi)$ and $V\cap h(\Phi)$.  $\Box$

\section{Appendix}
In this section we relax the assumption that spaces are locally compact separable metric.

 In our paper \cite{kktv} we considered the dimension $D_{\mathcal K}$ which unifies the covering and the cohomological dimensions.
 \begin{dfn}\emph{\cite{dr}}
 A sequence $\mathcal{K}=\{K_0,K_1,..\}$ of $CW$-complexes is called a {\em stratum} for a
dimension theory if
for each (paracompact) space $X$ admitting a perfect map onto a metrizable space, $K_n\in AE(X)$
implies both $K_{n+1}\in AE(X\times [0,1])$ and $K_{n+j}\in AE(X)$ for all $j\geq 0$
(here, $K_n\in AE(X)$ means
that $K_n$ is an absolute extensor for $X$).

Given a stratum $\mathcal{K}$, the dimension function
$D_{\mathcal{K}}$ for a  space $X$ as above is defined in the following way:
\begin{enumerate}
\item
$D_{\mathcal{K}}(X)=-1$ iff $X=\emptyset$;
\item $D_{\mathcal{K}}(X)\le n$ if
$K_n\in AE(X)$ for $n\ge 0$; if $D_{\mathcal{K}}(X)\le n$ and $K_m\not\in AE(X)$  for all $m<n$, then $D_{\mathcal{K}}(X)= n$;
\item
 $D_{\mathcal{K}}(X)=\infty$ if $D_{\mathcal{K}}(X)\le n$ is not satisfied for any $n$.
\end{enumerate}
\end{dfn}

   We proved in \cite[Theorem 2.6]{kktv} that any compact space $X$ (not necessarily metrizable) of dimension $D_{\mathcal K}X=n$ contains a continuum $V^n$ with respect to $D_{\mathcal K}$. The definition of continuum $V^n$ with respect to $D_{\mathcal K}$ is the same as that one of a  $V^n(G)$-space; the only difference is that we consider the dimension $D_{\mathcal K}$ instead of $\dim_G$. So, in order to prove Theorem~\ref{t1} in a more general version for the dimension $D_{\mathcal K}$ ,
we can use Theorem 2.6 from \cite{kktv} instead of Proposition 2.1.

 There is one claim from the proof of \cite[Theorem 2.6]{kktv}, see Claim 2.7, which was left without proof in that paper. Below, we  provide a detailed proof of the claim for completeness.
\begin{lem}\emph{\cite[Claim 2.7]{kktv}}
Let $F:Z\to K$ be a map, where $Z$ is a nonempty compact space and $K$ is a nonempty metrizable $ANR$.
Then for every open cover $\gamma$ of $K$ there exists an open cover $\nu$ of $Z$ satisfying the following condition: for any
closed set $B\subset Z$ and any $\nu$-map $\varphi: B\to Y$, where $Y$ is a paracompact space, there is a map $g:\varphi(B)\to K$ such that $g\circ\varphi$ is $\gamma$-close to the restriction $F|B$.
\end{lem}
\begin{proof}
We embed $K$ as a closed subset of a normed space $L$ and find a retraction $r:W\to K$, where $W$ is open in $L$. For every $x\in K$, choose $U_x\in\gamma$ containing $x$, and
 a convex open subset $\widetilde U_x$ of $L$ such that  $x\in \widetilde U_x$ and  $r(\widetilde U_x)\subset U_x$. Next, take an open cover $\gamma'$ of $K$ which  is a star-refinement of the cover $\widetilde\gamma=\{\widetilde U_x\cap K:x\in K\}$.

 Let us show that the cover $\nu=F^{-1}(\gamma')$ is as required.
Suppose $B\subset Z$ is closed and $\varphi:B\to Y$ is a $\nu$-map. We can assume that $\varphi(B)=Y$, so there is a finite open cover $\beta$ of $Y$ such that $\varphi^{-1}(\beta)$ refines $\nu$.  For every $V\in\beta$ fix $\Lambda_V\in\nu$ and a point $a_V\in F(\Lambda_V)$ with $\varphi^{-1}(V)\subset \Lambda_V$.
Let $N_\beta$ be the nerve of $\beta$ and $\kappa_\beta:Y\to N_\beta$ the canonical map assigning to each $y\in Y$ the point
$\sum_{V\in\beta} f_{V}(y)V$, where $y\in\cap_{V\in\beta} V$ and $\{f_V:V\in\beta\}$ is a partition of unity subordinated to $\beta$
(here the elements $V\in\beta$ are considered also as vertices of $N_\beta$).

We define a map
$h:N_\beta\to K$ by $$h(\sum_{V\in\beta} t_V V)=r(\sum_{V\in\beta} t_Va_V),\quad\text{where}\quad \sum_{V\in\beta}t_V=1, t_V\ge 0,$$  and let $g=h\circ\kappa_\beta$.
To check that $g\circ\varphi$ is $\gamma$-close to $F|B$, we fix $z\in B$. Then
$$\kappa_\beta(\varphi(z))=\sum_{V\in\beta} f_{V}(\varphi(z))V \quad\text{and}\quad
h(\kappa_\beta(\varphi(z)))=r(\sum_{V\in\beta} f_{V}(\varphi(z))a_V),$$ where $\varphi(z)\in\cap_{V\in\beta} V.$

Since $\gamma'$ is a star-refinement of $\widetilde\gamma$, there is $U_x\in\gamma$ such that $\widetilde U_x\cap K$ contains all points
$a_V$. So, $\sum_{V\in\beta} f_{V}(\varphi(z))a_V\in\widetilde U_x$ because $\widetilde U_x$ is convex. Consequently,
$h(\kappa_\beta(\varphi(z)))\in U_x$. On the other hand, $F(z)\in\cap_{V\in\beta} F(\Lambda_{V})\subset U_x$. Therefore, $F|B$ and $g\circ\varphi$ are $\gamma$-close.
\end{proof}

The proof of Theorem~\ref{c1} is based on Theorem 3.1 from \cite{tv}.
Here is the formulation of that theorem which was actually established in~\cite{tv}:
\begin{thm}\label{t3}\cite[Theorem 3.1]{tv}
Let $(X,F)$ be a strong $K^n_G$-manifold and $M$ be a Lindel\"{o}f, normally placed subset of $X$ with $\dim_GM\leq n-2$. Then for every  disjoint, nonempty, open subsets $U$ and $V$ of $X$ there exists a continuum $K\subset X\backslash M$ such that $U\cap K\neq\varnothing\neq V\cap K$.
\end{thm}
Recall that a set $M\subset X$ is normally placed if every two closed disjoint sets in $M$ have disjoint open in $X$ neighborhoods. For example, every $F_\sigma$-subset of a normal space is normally placed. The same is true if $X$ is a metric space and $M\subset X$ is arbitrary.


\end{document}